\documentclass{amsart}
\usepackage{amsmath,amsfonts,amssymb,enumerate,mathrsfs,mathtools,amscd}

\newcommand{\N}{\mathbb{N}}                   
\newcommand{\Z}{\mathbb{Z}}                   
\newcommand{\R}{\mathbb{R}}                   
\newcommand{\B}{\mathbb{B}}                   
\newcommand{\AR}{\mathscr{A{\!}R}}            

\newcommand{\Reg}{\mathrm{Reg}}
\newcommand{\Sing}{\mathrm{Sing}}
\newcommand{\Int}{\mathrm{Int}}

\newcommand{\vp}{\varphi}

\newcommand{\An}{\mathscr{A}}                 
\newcommand\rk{\mathrm{rk}}                   

\theoremstyle{plain}
\newtheorem{theorem}{Theorem}[section]
\newtheorem{proposition}[theorem]{Proposition}
\newtheorem{lemma}[theorem]{Lemma}
\newtheorem{corollary}[theorem]{Corollary}

\theoremstyle{definition}

\newtheorem{remark}[theorem]{Remark}

\numberwithin{equation}{section}

\begin{document}

\title[Globally subanalytic arc-symmetric sets]{Globally subanalytic arc-symmetric sets}

\author{Janusz Adamus}
\address{Department of Mathematics, The University of Western Ontario, London, Ontario, Canada N6A 5B7}
\email{jadamus@uwo.ca}
\thanks{J. Adamus's research was partially supported by the Natural Sciences and Engineering Research Council of Canada}

\subjclass[2010]{32B20, 14P15, 14E15, 32C40}
\keywords{subanalytic geometry, globally subanalytic set, arc-symmetric set, Nash subanalytic set, semianalytic set, arc-analytic function}

\begin{abstract}
It is shown that every C-semianalytic arc-symmetric set can be realized as the zero locus of an arc-analytic function. As a consequence, a Nash globally subanalytic arc-symmetric set is the zero locus of a continuous globally-subanalytic function which is arc-analytic outside a simple normal crossings divisor.
\end{abstract}
\maketitle


\section{Introduction}
\label{sec:intro}

The purpose of this article is to initiate a systematic study of a certain important class of subanalytic sets, which are closed under analytic continuation.
Let us begin by recalling some basic notions.
A set $X\subset\R^n$ is called \emph{semianalytic}, when every point $x\in\R^n$ has an open neighbourhood $U$ such that $X\cap U$ is a finite union of sets of the form
\[
\{y\in U\;:\;f(y)=0, g_1(y)>0,\dots,g_k(y)>0\}\,,
\]
where $f,g_1,\dots,g_k\in\An(U)$ are real analytic functions on $U$. A set $Y\subset\R^n$ is called \emph{subanalytic}, when for every point $x\in\R^n$ there are an open neighbourhood $U$ and a bounded semianalytic set $X\subset\R^{n+m}$, for some $m$, such that $Y\cap U=\pi(X)$, where $\pi:\R^{n+m}\to\R^n$ is the coordinate projection.

For any $n\in\Z_+$, let $v_n:\R^n\to\R^n$ denote the semialgebraic map
\[
(x_1,\dots,x_n)\ \mapsto\ \left(\frac{x_1}{\sqrt{1+x_1^2}},\dots,\frac{x_n}{\sqrt{1+x_n^2}}\right)\,.
\]
We say that a set $E\subset\R^n$ is \emph{globally subanalytic} if its image $v_n(E)$ is subanalytic in $\R^n$. Since $v_n$ is an analytic isomorphism onto the bounded open set $(-1,1)^n$, it follows that globally subanalytic sets are subanalytic. The importance of the class of globally subanalytic sets stems from the fact that they form an o-minimal structure $(\mathcal{S}_n)_{n\in\N}$ (see \cite{vDDM}). This means, by definition, that for every $n\in\N$, (i) $\mathcal{S}_n$ is a boolean algebra of subsets of $\R^n$, (ii) $X\in\mathcal{S}_n$ implies $X\times\R,\R\times X\in\mathcal{S}_{n+1}$, (iii) $\{(x_1,\dots,x_n):x_1=x_n\}\in\mathcal{S}_n$, (iv) $X\in\mathcal{S}_{n+1}$ implies $\pi(X)\in\mathcal{S}_n$, where $\pi:\R^{n+1}\to\R^n$ is the coordinate projection, (v) the set $\{(x,y)\in\R^2:x<y\}$ is in $\mathcal{S}_2$, and (vi) the only elements of $\mathcal{S}_1$ are the finite unions of open intervals and singletons.
O-minimality is responsible for several finiteness properties that we use throughout the paper, such as existence of finite stratifications. For details on o-minimal structures, see \cite{vDD}.

Finally, a set $E\subset\R^n$ is called \emph{arc-symmetric}, when for every real analytic arc $\gamma:(-1,1)\to\R^n$ with $\Int(\gamma^{-1}(E))\neq\varnothing$, one has $\gamma((-1,1))\subset E$. (Here and throughout, $\Int(S)$ (or $\Int_\Gamma(S)$) denotes the interior of a set $S$ (as a subset of $\Gamma$).) A function $f:\Omega\to\R$ on a real analytic manifold $\Omega$ is called \emph{arc-analytic}, when $f\circ\gamma$ is analytic for every real analytic arc $\gamma:(-1,1)\to\Omega$.
\medskip

Throughout this paper we shall denote by $\AR(\R^n)$ the class of globally subanalytic arc-symmetric subsets of $\R^n$. This class includes, in particular, the arc-symmetric semialgebraic subsets of $\R^n$ (introduced by Kurdyka \cite{Ku}). The geometry of semialgebraic arc-symmetric sets is well understood by now. In fact, it turns out that they are precisely the zero loci of semialgebraic arc-analytic functions. The elegant theory of semialgebraic arc-symmetric sets and arc-analytic functions provides a natural real counterpart of algebraic geometry over an algebraically closed field (see \cite{Ku, BM2, KP, AS1, AS2, KK} and the references therein).

Much less is known in the subanalytic setting. The main goal of this note is to show that certain fundamental analytic and geometric properties of the semialgebraic arc-symmetric sets can be generalized to our $\AR(\R^n)$. We begin by restricting to the class of C-semianalytic sets, introduced by Acquistapace, Broglia and Fernando \cite{ABF} (see Section~\ref{sec:ARC-sets}). Thanks to Theorem~\ref{thm:Nash-is-semi}, these can be viewed as a local model for a more general class, namely the Nash subanalytic arc-symmetric sets, studied in the final Section~\ref{sec:Nash-glob-sub}. Our main result, Theorem~\ref{thm:ARC-zero-set} asserts that locally a Nash subanalytic arc-symmetric set is the zero locus of a subanalytic arc-analytic function. As a consequence, globally such a set can be realized as the zero locus of a continuous globally subanalytic function, which is arc-analytic outside a simple normal crossings divisor (Theorem~\ref{thm:AR-zero-set}).


\section{$\AR$ topology}
\label{sec:AR-topology}

Let $\Omega$ be a connected real analytic submanifold of $\R^n$. We shall denote by $\AR(\Omega)$ the family of arc-symmetric subsets of $\Omega$ that are globally subanalytic as subsets of $\R^n$.

\begin{remark}
\label{rem:the-obvious}
{~}
\begin{enumerate}
\item Every $E\in\AR(\Omega)$ is a closed set in the Euclidean topology on $\Omega$. This follows from the subanalytic Curve Selection Lemma (see, e.g., \cite[1.17]{vDDM}).
\item $\AR(\R^n)$ contains all arc-symmetric semialgebraic subsets of $\R^n$.
\item $\AR(\Omega)$ contains globally subanalytic real analytic subsets of $\Omega$. Indeed, real analytic sets are arc-symmetric.
\end{enumerate}
\end{remark}

Kurdyka's fundamental \cite[Thm.\,1.4]{Ku} generalizes naturally to the globally subanalytic setting.

\begin{theorem}
\label{thm:AR-topology}
Let $\Omega$ be a connected, globally subanalytic, real analytic submanifold of $\R^n$.
There exists a noetherian topology on $\Omega$, whose closed sets are precisely the elements of $\AR(\Omega)$.
\end{theorem}

The theorem follows from the two lemmas below. We shall call the above noetherian topology the \emph{$\AR$-topology} on $\Omega$. The elements of $\AR(\Omega)$ will henceforth be called $\AR$-closed sets.

\begin{lemma}
\label{lem:dim-drop}
Let $\Gamma$ be a connected, smooth, subanalytic subset of $\Omega$, and let $E\subset\Omega$ be subanalytic and arc-symmetric. Then
\[
\Gamma\not\subset E\ \Longrightarrow\ \dim(\Gamma\cap E)<\dim\Gamma\,.
\]
\end{lemma}

The proof of Lemma~\ref{lem:dim-drop} is identical to that of \cite[1.6]{Ku}, as it only relies on basic topological properties of o-minimal structures. We include it for the reader's convenience.

\begin{proof}
Suppose that $\dim\Gamma\cap E=\dim\Gamma=k$. Then, $\Int_\Gamma(\Gamma\cap E)\neq\varnothing$, so one can pick a point $a\in\overline{\Int_\Gamma(\Gamma\cap E)}$. Let then $U$ be an open chart around $a$ in $\Gamma$ and let $\vp:U\to\B^k$ be an analytic isomorphism onto the open unit ball in $\R^k$ such that $\vp(a)=0$. We have $\vp(\Int_\Gamma(\Gamma\cap E))\cap\B^k\neq\varnothing$, and hence can pick a $b\in\vp(\Int_\Gamma(\Gamma\cap E))\cap\B^k$. Let now $x\in\B^k$ be arbitrary and let $\widetilde\gamma:[-1,1]\to\B^k$ be an analytic arc with $\widetilde\gamma(-1)=b$, $\widetilde\gamma(1)=x$. Set $\gamma\coloneqq\vp^{-1}\circ\widetilde\gamma$. Then, $\Int(\gamma^{-1}(E))\neq\varnothing$, and hence by arc-symmetry of $E$, $\gamma^{-1}(E)=[-1,1]$. In particular, $\vp^{-1}(x)\in E$. Since $x$ was arbitrary, we have $U\subset\Int_\Gamma(\Gamma\cap E)$, and so $a\in\Int_\Gamma(\Gamma\cap E)$. Since $a$ was arbitrary, this proves $\overline{\Int_\Gamma(\Gamma\cap E)}=\Int_\Gamma(\Gamma\cap E)$, and thus $\Gamma\cap E=\Gamma$, by connectedness of $\Gamma$.
\end{proof}

\begin{lemma}
\label{lem:fin-intersection}
Let $\Gamma$ be a globally subanalytic, smooth, connected subset of $\Omega$, and let $\{E_i\}_{i\in I}\subset\AR(\Omega)$. Then, there exist $i_1,\dots,i_s\in I$ such that
\[
\Gamma\cap\bigcap_{i\in I}E_i\ =\ \Gamma\cap E_{i_1}\cap\dots\cap E_{i_s}\,.
\]
\end{lemma}

The proof, again, is virtually identical to that of \cite[Lem.\,1.5]{Ku}. We include it for the reader's convenience.

\begin{proof}
We proceed by induction on $k=\dim\Gamma$. If $k=0$, then $\Gamma$ is a singleton and there is nothing to show. Suppose then that $k\geq1$ and the claim holds for all globally subanalytic smooth connected subsets of $\Omega$ of dimensions less than $k$. If $\Gamma\subset E_i$ for all $i\in I$, then again there is nothing to show, so let $i_0\in I$ be such that $\Gamma\cap E_{i_0}\not\supset\Gamma$. By Lemma~\ref{lem:dim-drop}, the globally subanalytic set $\Gamma\cap E_{i_0}$ is then of dimension less than or equal to $k-1$. By o-minimality, $\Gamma\cap E_{i_0}$ is a finite union of connected smooth globally subanalytic sets $\Gamma_1,\dots,\Gamma_s$. By induction, for each $j=1,\dots,s$, there exists a finite index subset $I_j\subset I$ such that $\Gamma_j\cap\bigcap_{i\in I}E_i=\Gamma_j\cap\bigcap_{i\in I_j}E_i$. Then,
\[
\Gamma\cap\bigcap_{i\in I}E_i\ =\ (\Gamma_1\cup\dots\cup\Gamma_s)\cap\bigcap_{i\in I}E_i\ =\ \bigcup_{j=1}^s(\Gamma_j\cap\bigcap_{i\in I_j}E_i)\ =\ 
\Gamma\cap\bigcap_{i\in I_1\cup\dots\cup I_s}E_i\,.
\]
\end{proof}

\begin{proof}[Proof of Theorem~\ref{thm:AR-topology}]
By Lemma~\ref{lem:fin-intersection}, letting $\Gamma=\Omega$, intersection of an arbitrary family of $\AR$-closed sets is an $\AR$-closed set. Clearly, finite unions of arc-symmetric sets are also arc-symmetric. So are the empty set $\varnothing$ and $\Omega$. Noetherianity of the $\AR$-topology follows from Lemma~\ref{lem:fin-intersection} again, since every decreasing sequence of $\AR$-closed sets stabilizes.
\end{proof}

Given a set $E\in\AR(\Omega)$, we will say that $E$ is \emph{$\AR$-irreducible} if $E$ cannot be expressed as a union of two proper $\AR$-closed subsets.
By noetherianity of $\AR$-topology, every $\AR$-closed set $E$ can be uniquely expressed as a finite union of $\AR$-irreducible sets
\[
E=E_1\cup\dots\cup E_s\,,\quad\mathrm{where\ }E_i\not\subset\bigcup_{j\neq i}E_j\mathrm{\ for\ all\ }i=1,\dots,s\,.
\]
The sets $E_1,\dots,E_s$ are called the \emph{$\AR$-irreducible components} of $E$.
By noetherianity of $\AR$-topology, one can also define the \emph{$\AR$-closure} of an arbitrary set $S\subset\Omega$, denoted $\overline{S}^\AR$, as the smallest (with respect to inclusion) $\AR$-closed subset of $\Omega$ that contains $S$.


\section{C-semianalytic arc-symmetric sets}
\label{sec:ARC-sets}

Let $\Omega$ be a real analytic manifold, and let $\Omega^*$ denote its complexification (see, e.g., \cite{GMT} for a modern exposition of complexification of real analytic spaces). A set $R\subset\Omega$ is called \emph{C-analytic}, when there exists an open neighbourhood $V^*$ of $\Omega$ in $\Omega^*$ and a complex analytic set $Z$ in $V^*$ such that $Z\cap\Omega=R$ (see, e.g., \cite{WB}). By \cite[Prop.\,10]{WB} (cf. \cite[Prop.\,15]{Car}), $R$ is a C-analytic subset of $\Omega$ if and only if $R$ can be realized as the common zero locus of finitely many real analytic functions on $\Omega$, and thus $R=f^{-1}(0)$ for some $f\in\An(\Omega)$.

For a set $S\subset\Omega$, its \emph{C-analytic closure} is the smallest C-analytic set in $\Omega$ which contains $S$. It is well defined, as the intersection of any family of C-analytic sets is itself C-analytic (see, e.g., \cite[\S\,8]{WB}).

Following \cite{ABF}, we will say that a subset $E$ of $\Omega$ is \emph{C-semianalytic}, when $E$ is a union of a locally finite family of global basic semianalytic subsets of $\Omega$, that is, sets of the form $\{f=0, g_1>0,\dots, g_s>0\}$, where $f, g_j\in\An(\Omega)$.
\smallskip

Let $\AR_C(\R^n)$ denote the family of C-semianalytic globally subanalytic arc-symmetric sets in $\R^n$. More generally, for a real analytic submanifold $\Omega\subset\R^n$, we shall denote by $\AR_C(\Omega)$ the family of C-semianalytic sets $E\subset\Omega$ such that $E$ is arc-symmetric in $\Omega$ and globally subanalytic as a subset of $\R^n$. 

\begin{theorem}
\label{thm:ARC-topology}
Let $\Omega$ be a connected globally subanalytic real analytic submanifold of $\R^n$.
There exists a noetherian topology on $\Omega$, whose closed sets are precisely the elements of $\AR_C(\Omega)$.
\end{theorem}

\begin{proof}
Let $\Gamma$ be a globally subanalytic smooth connected subset of $\Omega$, and let $\{E_i\}_{i\in I}\subset\AR_C(\Omega)$ be arbitrary. By Lemmas~\ref{lem:dim-drop} and~\ref{lem:fin-intersection},
\[
\Gamma\cap\bigcap_{i\in I}E_i\ =\ \Gamma\cap E_{i_1}\cap\dots\cap E_{i_s}\,,
\]
for some $i_1,\dots,i_s\in I$. By \cite[Prop.\,5.3.5]{ABF}, locally finite unions and intersections of C-semianalytic sets are C-semianalytic.
The claim thus follows as in the proof of Theorem~\ref{thm:AR-topology}.
\end{proof}

Given a set $E\in\AR_C(\Omega)$, we will say that $E$ is \emph{$\AR_C$-irreducible} if $E$ cannot be expressed as a union of two proper $\AR_C$-closed subsets.
By noetherianity of $\AR_C$-topology, every $\AR_C$-closed set $E$ can be uniquely expressed as a finite union of $\AR_C$-irreducible sets
\[
E=E_1\cup\dots\cup E_s\,,\quad\mathrm{where\ }E_i\not\subset\bigcup_{j\neq i}E_j\mathrm{\ for\ all\ }i=1,\dots,s\,.
\]
The sets $E_1,\dots,E_s$ are called the \emph{$\AR_C$-irreducible components} of $E$.
\smallskip

\begin{proposition}
\label{prop:sub-C-analytic}
Let $\Omega$ be a connected real analytic submanifold of $\R^n$.
Let $E\in\AR_C(\Omega)$ and let $R_E\subset\Omega$ denote its C-analytic closure. Then, $\dim{R_E}=\dim{E}$.
\end{proposition}

\begin{proof}
Let $k=\dim{E}$.
Recall that $R_E$ is the intersection of all C-analytic sets $R\subset\Omega$, such that $E\subset R$. Clearly, for every such $R$, $\dim{R}\geq\dim{E}$.
Therefore, it suffices to find a C-analytic $R$ with $R\supset E$ and $\dim{R}=k$.

Assume without loss of generality that $E\neq\varnothing$. As a globally subanalytic set, $E=\Gamma_1\cup\dots\cup\Gamma_s$ is a finite union of connected smooth globally subanalytic sets. Let $x_0\in E$ be arbitrary, and let $B$ be an open ball containing $x_0$, such that $E\cap B=(S_1\cup\dots\cup S_t)\cap B$, where each $S_j$ is a global basic semianalytic set in $\Omega$. Let $R_j$ be the C-analytic closure of $S_j$, $j=1,\dots,t$.
By \cite[Def.\,5.4.1]{ABF} and the following remarks, we have $\dim{R_j}=\dim{S_j}$, and hence there exists $j_1$ such that $\dim_{x_0}{R_{j_1}}=\dim_{x_0}{E}$. Then, $R_{j_1}$ contains a nonempty open subset of a manifold $\Gamma_{i_1}$ of dimension $\dim_{x_0}{E}$ adherent to $x_0$, for some $1\leq i_1\leq s$. Hence, $R_{j_1}\supset\Gamma_{i_1}$, by arc-symmetry (Lemma~\ref{lem:dim-drop}). Since the collection $\Gamma_1,\dots,\Gamma_s$ is finite, it follows that $E$ is contained in the union of finitely many C-analytic sets $R_{j_1},\dots,R_{j_r}$. Set $R=R_{j_1}\cup\dots\cup R_{j_r}$. Then, $R$ is C-analytic and $\dim{R}=\max_j\dim{R_j}=\max_i\dim{\Gamma_i}=\dim{E}$.
\end{proof}

Let now $R\subset\Omega$ be a C-analytic set of dimension $k>0$. By \cite[Prop.\,10]{WB}, $R$ is the zero locus of a coherent sheaf of $\An(\Omega)$-ideals. It thus admits a resolution of singularities $\pi:\widetilde{R}\to R$, where $\widetilde{R}$ is smooth of dimension $k$, and $\pi$ is a composite of a locally finite sequence of blowings-up with smooth centres (see \cite[Thm.\,13.4]{BM3}). Moreover, there is a C-analytic set $S\subset R$, of dimension $\dim{S}<k$, such that $R\setminus S$ is smooth of pure dimension $k$ and $\pi$ is an isomorphism over $R\setminus S$. (Indeed, if $Z$ is a $k$-dimensional complex analytic set in an open neighbourhood $V^*$ of $\Omega$ in $\Omega^*$ such that $Z\cap\Omega=R$ and $Z=Z'\cup Z''$ is its decomposition into analytic sets, where $Z'$ is of pure dimension $k$ and $\dim{Z''}<k$, then one can take $S=(\Sing(Z')\cup Z'')\cap\Omega$.)

\begin{lemma}
\label{lem:image-arc-sym}
Let $\Omega$ be a connected real analytic submanifold of $\R^n$.
Let $R\subset\Omega$ be a C-analytic set of dimension $k>0$, let $\pi:\widetilde{R}\to R$ be its resolution of singularities, and let $S\subset R$ be a C-analytic set of dimension $\dim{S}<k$, such that $R\setminus S$ is smooth of pure dimension $k$ and $\pi$ is an isomorphism over $R\setminus S$. Then, for every connected component $\widetilde{E}$ of $\widetilde{R}$, the set $\pi(\widetilde{E})\cup S$ is arc-symmetric in $\Omega$.
\end{lemma}

\begin{proof}
Let $\widetilde{E}$ be a connected component of $\widetilde{R}$, and let $\gamma:(-1,1)\to R$ be an analytic arc with $\Int(\gamma^{-1}(\pi(\widetilde{E})\cup S))\neq\varnothing$. If $\Int(\gamma^{-1}(S))\neq\varnothing$, then $\gamma((-1,1))\subset S$, by arc-symmetry of $S$. Otherwise, $\gamma$ intersects $S$ only at isolated points, and hence there is a unique analytic arc $\widetilde\gamma:(-1,1)\to\widetilde{R}$ such that $\gamma=\pi\circ\widetilde\gamma$. It now follows that $\Int(\widetilde\gamma^{-1}(\widetilde{E}))=\Int(\gamma^{-1}(\pi(\widetilde{E}))\neq\varnothing$, and hence $\widetilde\gamma((-1,1))\subset \tilde{E}$, by arc-symmetry of $\widetilde{E}$. Consequently, $\gamma((-1,1))\subset\pi(\widetilde{E})$, which completes the proof.
\end{proof}

\begin{proposition}
\label{prop:covers-reg-k}
Let $\Omega$ be a connected real analytic submanifold of $\R^n$.
Let $E\subset\Omega$ be an $\AR$-closed set of dimension $k>0$, and let $R\subset\Omega$ be its C-analytic closure. Assume that $\dim{R}=k$. Let $\pi:\widetilde{R}\to R$ be a desingularization of $R$. Then, there exist finitely many connected components $\widetilde{E}_1,\dots,\widetilde{E}_t$ of $\widetilde{R}$ of dimension $k$, such that
\[
\overline{\Reg_k(E)}\subset\pi(\widetilde{E}_1\cup\dots\cup\widetilde{E}_t)\subset E\,.
\]
\end{proposition}

\begin{proof}
Let $S\subset R$ be a C-analytic set of dimension $\dim{S}<k$, such that $R\setminus S$ is a $k$-dimensional real analytic manifold and $\pi$ is an isomorphism over $R\setminus S$. Let $\Gamma_1,\dots,\Gamma_s$ be the connected components of $\Reg_k(E)$ (finitely many, by global subanalyticity of $E$). For every $j=1,\dots,s$, $\Gamma_j\not\subset S$. As $\pi$ is an isomorphism over $R\setminus S$, there exists a connected component $\widetilde{E}_j$ of $\widetilde{R}$, of dimension $k$, such that $\pi(\widetilde{E}_j)$ contains a nonempty open subset of $\Gamma_j$. Then, $\Gamma_j\subset\pi(\widetilde{E}_j)\cup S$, by Lemmas~\ref{lem:dim-drop} and~\ref{lem:image-arc-sym}. Note that $\pi(\widetilde{E}_j)$ is a closed set, as $\pi$ is proper and $\widetilde{E}_j$ is a closed subset of $\widetilde{R}$. Since $\Gamma_j\cap S$ is nowhere dense in $\Gamma_j$, it follows that $\Gamma_j\subset\pi(\widetilde{E}_j)$.

On the other hand, the subanalytic arc-symmetric set $\pi^{-1}(E)$ contains a nonempty open subset of the connected manifold $\widetilde{E}_j$, and so $\widetilde{E}_j\subset\pi^{-1}(E)$ and $\pi(\widetilde{E}_j)\subset E$. It follows that
\[
\overline{\Reg_k(E)}\subset\overline{\pi(\widetilde{E}_1)\cup\dots\cup\pi(\widetilde{E}_s)}=\pi(\widetilde{E}_1\cup\dots\cup\widetilde{E}_s)\subset E\,,
\]
as required. (Note that there may be some repetitions among the $\widetilde{E}_1,\dots,\widetilde{E}_s$.)
\end{proof}

By noetherianity of $\AR_C$-topology, one can also define the \emph{$\AR_C$-closure} of an arbitrary set $S\subset\Omega$, denoted $\overline{S}^{\AR_C}$, as the smallest (with respect to inclusion) $\AR_C$-closed subset of $\Omega$ that contains $S$.

\begin{remark}
\label{rem:no-dim-preservation}
Unfortunately, in the subanalytic context, the $\AR$- and $\AR_C$-closure behaves in a much less controlled way than in the semialgebraic setting of \cite{Ku}. In particular, for an arbitrary globally subanalytic set $S$, one may have $\dim\overline{S}^{\AR}>\dim{S}$.
Indeed, consider, for example, $S=\{(x,y)\in\R^2\!:y=\sin{x},$ $-1\leq x\leq1\}$. Then, $S$ is globally subanalytic in $\R^2$ as a bounded subanalytic set, however by analytic continuation any arc-symmetric set in $\R^2$ containing $S$ must contain the whole graph of the sine function as well. Thus, $\overline{S}^\AR=\R^2$, hence also $\overline{S}^{\AR_C}=\R^2$.
Note that $S$ is also C-semianalytic, since $\sin{x}$ is a global real analytic function on $\R$.
\end{remark}

Nonetheless, the topological dimension (as a subanalytic set) of any $\AR_C$-closed set coincides with its $\AR_C$-Krull dimension, at least in the compact setting, as shown below.
For a non-empty $\AR_C$-closed set $E$ we define its \emph{Krull dimension} as
\[
\dim_K{E}=\sup\{l\in\N\;:\;\exists E_0\varsubsetneq E_1\varsubsetneq\dots\varsubsetneq E_l\subset E,\ \ \mathrm{where\ }E_j\mathrm{\ are\ }\AR_C\mathrm{-irreducible}\}\,.
\]
By convention, $\dim\varnothing=\dim_K\varnothing=-1$.

\begin{theorem}
\label{thm:AR-equals-Krull}
Let $\Omega$ be a compact connected real analytic submanifold of $\R^n$.
If $E$ is a non-empty $\AR_C$-closed set in $\Omega$, then
\[
\dim_K\!{E}\;=\;\dim{E}\,,
\]
where $\dim{E}$ is the supremum of dimensions of real analytic submanifolds of $E$.
\end{theorem}

\begin{proof}
By Proposition~\ref{prop:ARC-dim-drops} below, we have $\dim_K\!E\leq\dim{E}$.
For the proof of the other inequality, we proceed by induction on $k=\dim{E}$. The base case being clear, assume $k\geq1$.
By the Good Directions Lemma in o-minimal structures (see \cite[4.9]{vDDM} or \cite[Thm.\,VII.4.2]{vDD}), there is a $k$-dimensional linear subspace $U$ of $\R^n$ such that the orthogonal projection $\pi:\R^n\to U$ has finite fibres when restricted to $E$. Suppose $U$ is spanned by vectors $u_1,\dots,u_k$ in $\R^n$, and let $V=\mathrm{span}\{u_2,\dots,u_k\}$. Then, the set $F=E\cap\pi^{-1}(V)$ is $\AR_C$-closed as the intersection of two $\AR_C$-closed sets, and of dimension $k-1$. By the finiteness of decomposition into $\AR_C$-irreducible components, at least one such component of $F$ is of dimension $k-1$.
\end{proof}


\section{Compact submanifold setting}
\label{sec:rel-compact}

In this section, we assume that $\Omega$ is a compact connected real analytic submanifold of $\R^n$.
Then, real analytic subsets of $\Omega$ are globally subanalytic, and hence C-analytic subsets of $\Omega$ are $\AR_C$-closed.

\begin{proposition}
\label{prop:ARC-irred-C-irred}
Let $E\in\AR_C(\Omega)$ be $\AR_C$-irreducible. Then, there exists a unique C-irreducible C-analytic set $R\subset\Omega$, such that $E\subset R$, $\dim{R}=\dim{E}$, and $R\subset S$ for every C-analytic set $S\subset\Omega$ containing $E$.
\end{proposition}

\begin{proof}
Let $R$ be the C-analytic closure of $E$. By Proposition~\ref{prop:sub-C-analytic}, $R$ is of dimension $\dim{E}$.
Moreover, $R$ is C-irreducible, for if $R=R_1\cup R_2$ for some proper C-analytic subsets $R_1,R_2$, then $E=(E\cap R_1)\cup(E\cap R_2)$ is $\AR_C$-reducible.
\end{proof}

\begin{lemma}
\label{lem:ARC-closure-dim}
Let $E\subset\Omega$ be an $\AR_C$-closed set of dimension $k>0$, and let $E=\Gamma_1\cup\dots\cup\Gamma_s$ be its partition into smooth connected globally subanalytic sets. If $j\in\{1,\dots,s\}$ is such that $\Gamma_j\not\subset\overline{\Reg_k(E)}$, then $\dim\overline{\Gamma_j}^{\AR_C}=\dim\Gamma_j<k$.
\end{lemma}

\begin{proof}
Pick a point $x_0\in\Gamma_j\setminus\overline{\Reg_k(E)}$. Let $B$ be an open ball centered at $x_0$ such that $B\cap\overline{\Reg_k(E)}=\varnothing$. Since $E$ and $B$ are C-semianalytic, there exists a C-analytic $R\subset\Omega$, of dimension $\dim_{x_0}{\Gamma_j}$, such that $\Gamma_j\cap B\subset R$. Then, $\Gamma_j\subset R$, by Lemma~\ref{lem:dim-drop}, and so $\overline{\Gamma_j}^{\AR_C}\subset R$, as $R$ is $\AR_C$-closed. Consequently, $\dim\overline{\Gamma_j}^{\AR_C}\leq\dim{R}=\dim{\Gamma_j}<k$.
\end{proof}

\begin{corollary}
\label{cor:subset-of-irred}
Let $E\in\AR_C(\Omega)$ be $\AR_C$-irreducible, of dimension $k>0$. If $F\in\AR_C(\Omega)$ and $\Reg_k(E)\subset F\subset E$, then $F=E$.
\end{corollary}

\begin{proof}
If $F\supset\Reg_k(E)$, then $F\supset\overline{\Reg_k(E)}$, since arc-symmetric sets are closed in Euclidean topology. Let $E=\Gamma_1\cup\dots\cup\Gamma_s$ be a  finite partition of $E$ into smooth connected globally subanalytic sets. If $F\neq E$, there exists $j\in\{1,\dots,s\}$ such that $\Gamma_j\not\subset F$. Then, $\Gamma_j\not\subset\overline{\Reg_k(E)}$. Let $\Gamma_{j_1},\dots,\Gamma_{j_q}$ be all such strata. By Lemma~\ref{lem:ARC-closure-dim}, $\dim\overline{\Gamma_{j_1}\cup\dots\cup\Gamma_{j_q}}^{\AR_C}<k$, and hence $\overline{\Gamma_{j_1}\cup\dots\cup\Gamma_{j_q}}^{\AR_C}\neq E$. It follows that $E=F\cup\overline{\Gamma_{j_1}\cup\dots\cup\Gamma_{j_q}}^{\AR_C}$ is a decomposition of $E$ into nonempty proper $\AR_C$-closed subsets; a contradiction.
\end{proof}

In the compact setting, we can refine Proposition~\ref{prop:covers-reg-k} as follows. This is a C-semianalytic analog of \cite[Thm.\,2.6]{Ku}.

\begin{theorem}
\label{thm:smooth-covers-all}
Let $E\in\AR_C(\Omega)$ be an $\AR_C$-irreducible set of dimension $k>0$, and let $R\subset\Omega$ be its C-analytic closure. Let $\pi:\widetilde{R}\to R$ be a desingularization of $R$. Then, there exists a unique connected component $\widetilde{E}$ of $\widetilde{R}$ of dimension $k$, such that
\[
\overline{\Reg_k(E)}\subset\pi(\widetilde{E})\subset E\,.
\]
\end{theorem}

\begin{proof}
Let $S\subset R$ be a C-analytic set of dimension $\dim{S}<k$, such that $\pi$ is an isomorphism over $R\setminus S$, and $R\setminus S$ is a $k$-dimensional real analytic manifold. Let $\{C_\lambda\}_{\lambda\in\Lambda}$ be the locally finite family of its connected components. By \cite[Prop.\,5.3.5]{ABF}, each $C_\lambda$ is C-semianalytic.

Note that, by Lemma~\ref{lem:dim-drop}, for every $\lambda\in\Lambda$, $E\supset C_\lambda$ or else $E\cap C_\lambda$ is nowhere dense in $C_\lambda$ and $\dim_xE<k$ for every $x\in E\cap C_\lambda$. Since $\dim{E\cap S}<k$, $\Reg_k(E)$ contains a nonempty open subset of $R\setminus S$, and so there is at least one $\lambda$ such that $E\supset C_\lambda$. It follows that there exists a nonempty $\Lambda_E\subset\Lambda$ such that
\begin{equation}
\label{eq:R-S-decomp}
E\cap(R\setminus S)=\bigcup_{\lambda\in\Lambda_E}\!\!C_\lambda\cup(\bigcup_{\lambda\in\Lambda\setminus\Lambda_E}\!\!\!C_\lambda\cap E_{(<k)})\,,
\end{equation}
where $E_{(<k)}=\{x\in E:\dim_xE\leq k-1\}$. Moreover, by \cite[Prop.\,5.3.8]{ABF}, the set $E_{(<k)}$ and hence all the disjoint summands of the right hand side of~\eqref{eq:R-S-decomp} are C-semianalytic.

Let $\Gamma_1,\dots,\Gamma_s$ be the connected components of $\Reg_k(E)$. For every $j=1,\dots,s$, let $\Lambda_j=\{\lambda\in\Lambda: C_\lambda\subset\Gamma_j\}$. Then, \[
\Gamma_j=\bigcup_{\lambda\in\Lambda_j}C_\lambda\cup(\Gamma_j\cap S)\,,
\]
and as $\Gamma_j\cap S$ is nowhere dense in $\Gamma_j$, then
\[
\overline{\Gamma_j}=\overline{\bigcup_{\lambda\in\Lambda_j}C_\lambda}=\bigcup_{\lambda\in\Lambda_j}\overline{C_\lambda}\,.
\]
Moreover, $\Lambda_E=\Lambda_1\cup\dots\cup\Lambda_s$, and hence
\[
\overline{\Reg_k(E)}=\overline{\Gamma_1\cup\dots\cup\Gamma_s}=\bigcup_{\lambda\in\Lambda_E}\overline{C_\lambda}\,.
\]

Let $\widetilde{E}_\delta$ be a connected component of $\widetilde{R}$, of dimension $k$. Let $\lambda\in\Lambda$ be such that $\pi(\widetilde{E}_\delta)$ contains a nonempty open subset of $C_\lambda$. Then, $C_\lambda\subset\pi(\widetilde{E}_\delta)\cup S$, by Lemmas~\ref{lem:dim-drop} and~\ref{lem:image-arc-sym}. Since $C_\lambda\cap S=\varnothing$, we obtain $C_\lambda\subset\pi(\widetilde{E}_\delta)$. In fact, $\overline{C_\lambda}\subset\pi(\widetilde{E}_\delta)$, since $\pi(\widetilde{E}_\delta)$ is closed in $R$ as the image of a closed set by a proper mapping. On the other hand, if $\pi(\widetilde{E}_\delta)$ contains no open subset of $C_\lambda$, then $\rk_x\pi<k$ for all $x\in\widetilde{E}_\delta\cap\pi^{-1}(C_\lambda)$, whence $\pi(\widetilde{E}_\delta)\cap C_\lambda\subset S$ and so $\pi(\widetilde{E}_\delta)\cap C_\lambda=\varnothing$, as $C_\lambda\subset R\setminus S$.
It thus follows from~\eqref{eq:R-S-decomp} that $(\pi(\widetilde{E}_\delta)\cup S)\cap E=(\pi(\widetilde{E}_\delta)\cap E\cap(R\setminus S))\cup(E\cap S)$ is C-semianalytic, and hence $\AR_C$-closed.

Let now $\lambda_0\in\Lambda_E$ be arbitrary and let $\widetilde{E}_0$ be a connected component of $\widetilde{R}$ satisfying $\pi(\widetilde{E}_0)\supset\overline{C_{\lambda_0}}$. We claim that then $\pi(\widetilde{E}_0)\supset C_\lambda$ for all $\lambda\in\Lambda_E$. Indeed, for else, if $\Lambda_0=\{\lambda\in\Lambda_E:C_\lambda\subset\pi(\widetilde{E}_0)\}$ is a proper subset of $\Lambda_E$, then
\[
\bigcup_{\lambda\in\Lambda_E\setminus\Lambda_0}\!\!\!C_\lambda\subset\pi(\bigcup_{\delta\in\Delta}\widetilde{E}_\delta)\,,
\]
for some family $\{\widetilde{E}_\delta\}_{\delta\in\Delta}$ of components of $\widetilde{R}$ different from $\widetilde{E}_0$. As the family $\{\pi(\widetilde{E}_\delta)\cap E\cap(R\setminus S)\}_{\delta\in\Delta}$ is locally finite, the set $(\pi(\bigcup_{\delta\in\Delta}\widetilde{E}_\delta)\cup S)\cap E$ is C-semianalytic, and hence $\AR_C$-closed.
We thus get a decomposition of $E$ into proper $\AR_C$-closed subsets
\[
E=[(\pi(\widetilde{E}_0)\cup S)\cap E]\cup[(\pi(\bigcup_{\delta\in\Delta}\widetilde{E}_\delta)\cup S)\cap E]\,,
\]
contradicting the $\AR_C$-irreducibility of $E$.

On the other hand, the arc-symmetric set $\pi^{-1}(E)$ contains a nonempty open subset of the manifold $\widetilde{E}_0$, and so $\widetilde{E}_0\subset\pi^{-1}(E)$ and $\pi(\widetilde{E}_0)\subset E$, which completes the proof.
\end{proof}

\begin{remark}
\label{rem:even-better}
Let $E\in\AR_C(\Omega)$ be an $\AR_C$-irreducible set. 
Let $R$, $\pi:\widetilde{R}\to R$, and $S\subset R$ be as above. Note that, by Lemma~\ref{lem:ARC-closure-dim} and the above proof, one actually gets that for a certain (unique) connected component $\widetilde{E}$ of $\widetilde{R}$
\[
E=(\pi(\widetilde{E})\cup S)\cap E\,.
\]
\end{remark}

\begin{proposition}
\label{prop:ARC-dim-drops}
Let $E,F\in\AR_C(\Omega)$, $F\varsubsetneq E$, and suppose that $E$ is $\AR_C$-irreducible of dimension $k>0$. Then, $\dim{F}<\dim{E}$.
\end{proposition}

\begin{proof}
Suppose to the contrary that $\dim{F}=k$. Let $R\subset\Omega$ be the C-analytic closure of $E$ and let $\pi:\widetilde{R}\to R$ be its desingularization. By Theorem~\ref{thm:smooth-covers-all}, there is a unique connected component $\widetilde{E}$ of $\widetilde{R}$, such that $\pi(\widetilde{E})\supset\Reg_k(E)$. Then, the subanalytic arc-symmetric set $\pi^{-1}(F)$ contains a nonempty open subset of $\widetilde{E}$, and hence $\widetilde{E}\subset\pi^{-1}(F)$, by Lemma~\ref{lem:dim-drop}. Consequently, $\pi(\widetilde{E})\subset F$, hence $F\supset\Reg_k(E)$, and so $F=E$, by Corollary~\ref{cor:subset-of-irred}; a contradiction.
\end{proof}

\begin{proposition}
\label{prop:smooth-locus}
For every $E\in\AR_C(\Omega)$ of dimension $k>0$, there exists $F\in\AR_C(\Omega)$ such that $\dim(E\cap F)<k$ and $E\setminus F$ is a $k$-dimensional manifold.
\end{proposition}

\begin{proof}
Let $R$ be the C-analytic closure of $E$. Then, $R$ is of dimension $k$, and there is a C-analytic set $S\subset\Omega$, of dimension $\dim{S}<k$, such that $R\setminus S$ is smooth of pure dimension $k$. Set $F=E\cap S$.
\end{proof}

\begin{remark}
\label{rem:no-equality}
Note that, in general, one cannot expect that $\Reg_k{E}=E\setminus F$ for some $\AR_C$-closed set $F$. Indeed, this may not be true even if $E$ is real algebraic.
\end{remark}
\medskip


\section{Arc-symmetric sets are zero loci of arc-analytic functions}
\label{sec:zero-set-thm}

Let $E\subset\R^n$ be non-empty. Recall that a function $f:E\to\R$ is called \emph{arc-analytic}, when $f\circ\gamma$ is an analytic function for every analytic arc $\gamma:(-1,1)\to E$. It is called a \emph{globally subanalytic function}, when the graph $\Gamma_f$ of $f$ is a globally subanalytic set in $\R^{n+1}$.

It is well known that every globally subanalytic arc-analytic function is continuous in the Euclidean topology (see, e.g., \cite[Lem.\,6.8]{BM2}). Moreover, by a straightforward adaptation of \cite[Prop.\,5.1]{Ku}, one has the following.

\begin{remark}
\label{rem:graph-AR-closed}
Let $E\in\AR(\R^n)$ be non-empty, and let $f:E\to\R^m$ be a globally subanalytic function whose all components are arc-analytic. Then
\begin{itemize}
\item[(i)] $\Gamma_f\in\AR(\R^n\times\R^m)$.
\item[(ii)] If $Z\in\AR(\R^m)$, then $f^{-1}(Z)\in\AR(\R^n)$.
\end{itemize}
\end{remark}

We are now ready to prove our main result. This is a C-semianalytic analog of \cite[Thm.\,1]{AS1} and the proof below is a direct adaptation of our argument from \cite{AS1}.

\begin{theorem}
\label{thm:ARC-zero-set}
Let $\Omega$ be a compact connected real analytic submanifold in $\R^n$, and let $E\in\AR_C(\Omega)$.
There exists a globally subanalytic arc-analytic function $f:\Omega\to\R$, such that $E=f^{-1}(0)$.
\end{theorem}

\begin{proof}
We argue by induction on dimension of $E$. If $\dim{E}\leq0$, then $E$ is a finite set, and hence the zero locus of a polynomial function on $\R^n$ restricted to $\Omega$. Suppose then that $\dim{E}=k>0$, and every $\AR_C$-closed set in $\Omega$ of dimension less than $k$ is the zero locus of a globally subanalytic arc-analytic function on $\Omega$. Assume without loss of generality that $E$ is $\AR_C$-irreducible.

Let $R\subset\Omega$ be the C-analytic closure of $E$. Then, $R$ is $k$-dimensional and C-irreducible, by Proposition~\ref{prop:ARC-irred-C-irred}.
Let $\pi:\widetilde{\Omega}\to\Omega$ be an embedded desingularization of $R$, and let $\widetilde{R}$ denote the strict transform of $R$ by $\pi$. Let $S\subset R$ be a C-analytic set, of dimension $\dim{S}<k$, such that $R\setminus S$ is smooth of pure dimension $k$, and $\pi$ is an isomorphism over $\Omega\setminus S$. Since $E\cap S$ is an $\AR_C$-closed set of dimension strictly less than $k$, the inductive hypothesis implies that there is a globally subanalytic arc-analytic function $h:\Omega\to\R$ such that $E\cap S=h^{-1}(0)$.

By Theorem~\ref{thm:smooth-covers-all}, there is a unique connected component $\widetilde{E}$ of $\widetilde{R}$ such that $\Reg_k(E)\subset\pi(\widetilde{E})\subset E$. Let $D=\pi^{-1}(S)$ and $Z=\widetilde{E}\cap D$. Let $\sigma:\widehat\Omega\to\widetilde\Omega$ be the blowing-up of $\widetilde\Omega$ at the C-analytic set $Z$. Let $\widehat{E}$ and $\widehat{D}$ denote the strict transforms of $\widetilde{E}$ and $D$ by $\sigma$, respectively. Since $\widetilde{E}$ and $D$ have only normal crossings (cf. \cite[Thm.1.6]{BM3}), $\widehat{E}$ and $\widehat{D}$ are disjoint subsets of $\widehat\Omega$. The sets $\widetilde{E}$ and $D$ are both C-analytic, and hence so are $\widehat{E}$ and $\widehat{D}$. We may thus choose non-negative analytic functions $v_1,v_2\in\An(\widehat\Omega)$, such that $v_1^{-1}(0)=\widehat{E}$ and $v_2^{-1}(0)=\widehat{D}$. Then, $v_1+v_2>0$ on $\widehat\Omega$ as $\widehat{E}\cap\widehat{D}=\varnothing$, and so $\vp\coloneq v_1/(v_1+v_2)$ defines an analytic function on $\widehat\Omega$. Note that $\vp\geq0$ on $\widehat\Omega$, $\vp|_{\widehat{E}}\equiv0$, and $\vp|_{\widehat{D}}\equiv1$.
Finally, set $v\coloneq v_1\!\cdot\!v_2$. Then, $v\in\An(\widehat\Omega)$, $v^{-1}(0)=\widehat{E}\cup\widehat{D}$, and $v\geq0$ on $\widehat\Omega$.

Now, define $\widehat{f}:\widehat\Omega\to\R$ by the formula
\[
\widehat{f}\coloneq(\vp\cdot(h\circ\pi\circ\sigma))^2+v^2\,.
\]
Note that $\widehat{f}$ is an arc-analytic function on $\widehat\Omega$, $\widehat{f}=(h\circ\pi\circ\sigma)^2$ on $\widehat{D}$, $\widehat{f}=0$ on $\widehat{E}$, and $\widehat{f}$ is strictly positive outside $\widehat{E}\cup\widehat{D}$.

Next, we compose $\widehat{f}$ with $\sigma^{-1}$ in order to get an arc-analytic function on $\widetilde\Omega$. More precisely, define $\widetilde{f}:\widetilde\Omega\to\R$ as
\[
\widetilde{f}(y)\coloneqq\begin{cases}((\widehat{f}\circ\sigma^{-1})\cdot(h\circ\pi))(y), &y\notin Z\\ 0, &y\in Z\,.\end{cases}
\]
To see that $\widetilde{f}$ is arc-analytic, let $\widetilde{\gamma}:(-1,1)\to\widetilde\Omega$ be an analytic arc not contained in $Z$, and let $\widehat\gamma:(-1,1)\to\widehat\Omega$ be its lifting by $\sigma$. Then, $\sigma\circ\widehat\gamma=\widetilde\gamma$. We claim that 
\begin{equation}
\label{eq:Gogamma}
\widetilde{f}\circ\widetilde{\gamma} = (\widehat{f}\circ\widehat\gamma) \cdot (h \circ \pi \circ \widetilde{\gamma})\,,
\end{equation}
which implies that $\widetilde{f}\circ \widetilde{\gamma}$ is analytic. Indeed, if $\widetilde{\gamma}(t) \notin Z$, then \eqref{eq:Gogamma} holds because $(\widehat{f}\circ\sigma^{-1}\circ\widetilde\gamma)(t)=(\widehat{f}\circ\sigma^{-1}\circ\sigma\circ\widehat\gamma)(t)=(\widehat{f}\circ\widehat\gamma)(t)$. If, in turn, $\widetilde{\gamma}(t) \in Z$, then $(h \circ \pi \circ \widetilde{\gamma})(t)=0$, by definition of $h$, and hence both sides of \eqref{eq:Gogamma} are equal to zero.

Now, we compose $\widetilde{f}$ with $\pi^{-1}$ to get an arc analytic function on $\Omega$. More precisely, we define $f:\Omega\to\R$ as
\[
f(x) \coloneqq\begin{cases}(\widetilde{f}\circ\pi^{-1})(x), &x\notin S\\ h^3(x), &x \in S\,.\end{cases}
\]
To see that $f$ is arc-analytic, let $\gamma:(-1,1)\to\Omega$ be an analytic arc not contained in $S$. Let $\widetilde\gamma:(-1,1)\to\widetilde\Omega$ be the lifting of $\gamma$ by $\pi$, and let $\widehat\gamma:(-1,1)\to\widehat\Omega$ be the lifting of $\widetilde\gamma$ by $\sigma$. Then, $\pi\circ \widetilde\gamma=\gamma$, and $\sigma\circ\widehat\gamma=\widetilde\gamma$. We claim that 
\begin{equation}
\label{eq:fogamma}
f\circ\gamma = \widetilde{f}\circ\widetilde\gamma\,,
\end{equation}
which implies that $f\circ\gamma$ is analytic. Indeed, if $\gamma(t) \not\in S$, then \eqref{eq:fogamma} holds because 
$(\widetilde{f}\circ\pi^{-1}\circ\gamma)(t)=(\widetilde{f}\circ\pi^{-1}\circ\pi\circ\widetilde\gamma)(t)=(\widetilde{f}\circ\widetilde\gamma)(t)$. If, in turn, $\gamma(t) \in S \cap \pi(\widetilde{E})$, then $h(\gamma(t))=0$ and hence $(f\circ\gamma)(t)=0$. But $\widetilde\gamma(t) \in Z$, and hence $(\widetilde{f}\circ\widetilde\gamma)(t)=0$ as well. Finally, if $\gamma(t)\in S\setminus\pi(\widetilde{E})$, then $\widetilde\gamma(t)\notin Z$ and $\widehat\gamma(t)\in\widehat{D}$; hence, by \eqref{eq:Gogamma}, we have
\begin{multline}
\notag
(\widetilde{f}\circ\widetilde\gamma)(t) = ((\widehat{f}\circ\widehat\gamma)\cdot(h\circ\pi\circ\widetilde\gamma))(t)
= \left(((h\circ\pi\circ\sigma)^2\circ\widehat\gamma)\cdot(h\circ\pi\circ\widetilde\gamma)\right)(t)\\
= \left((h\circ\pi\circ\widetilde\gamma)^2\cdot(h\circ\pi\circ\widetilde\gamma)\right)(t)
= (h\circ\pi\circ\widetilde\gamma)^3(t) = (h\circ\gamma)^3(t) = (f\circ\gamma)(t)\,.
\end{multline}
We shall now calculate the zero locus of $f$.
\begin{align*}
f^{-1}(0) & = \{x \in \Omega\setminus S : (\widetilde{f}\circ\pi^{-1})(x)=0 \} \cup \{x \in S : h^3(x)=0 \}\\
& = \pi\left(\{y \in \widetilde\Omega\setminus D : \widetilde{f}(y)=0\}\right) \cup (E \cap S)\\
& = \pi\left(\{y \in \widetilde\Omega\setminus D : ((\widehat{f}\circ\sigma^{-1})\cdot(h\circ\pi))(y)=0\}\right) \cup (E \cap S)\\
& = \pi\left(\{y \in \widetilde\Omega\setminus D : (\widehat{f}\circ\sigma^{-1})(y)=0\}\right) \cup (E \cap S)\\
& = \left((\pi\circ\sigma)(\{z\in\widehat\Omega\setminus\sigma^{-1}(D) :  \widehat{f}(z)=0 \})\right) \cup (E \cap S)\\
& = (\pi\circ\sigma)(\widehat{E}\setminus\sigma^{-1}(D)) \cup (E \cap S)\\
& = \pi(\widetilde{E}\setminus D) \cup (E \cap S) = \pi(\widetilde{E}) \cup (E \cap S) = (\pi(\widetilde{E})\cup S)\cap E\,.
\end{align*}
The latter set, by Remark~\ref{rem:even-better}, is equal to $E$, which completes the proof.
\end{proof}

\begin{corollary}
\label{cor:ARC-on-cubes}
Let $U$ be a non-empty open set in $\R^n$ and let $E\in\AR_C(U)$. Let $I_1,\dots,I_n\subset\R$ be closed intervals such that the cube $C=I_1\times\dots\times I_n$ is a subset of $U$. Then, there is a continuous globally subanalytic function $f:C\to\R$, which is arc-analytic on the interior of $C$ and such that $f^{-1}(0)=(E\cup\mathrm{fr}(C))\cap C$.
\end{corollary}

\begin{proof}
For simplicity of notation, assume without loss of generality that $C=[-1,1]^n$. Let $S^1\subset\R^2$ denote the unit circle. The projection $S^1\ni(x,y)\mapsto x\in[-1,1]$ induces a real analytic mapping $p:T^n\to C$ from the $n$-torus $T^n=(S^1)^n$ onto $C$. Let $D\subset T^n$ denote the inverse image of the boundary of $C$. Then, $p|_{T^n\setminus D}$ is a $2^n$-sheeted analytic covering of the open cube $(-1,1)^n$.

Set $F\coloneqq p^{-1}(E\cap C)$. Then, $F$ is an $\AR_C$-closed subset of the compact manifold $T^n$, and hence so is $F\cup D$. By Theorem~\ref{thm:ARC-zero-set}, there is a globally subanalytic arc-analytic function $g:T^n\to\R$ with $g^{-1}(0)=F\cup D$. Let $V\subset T^n$ be one of the connected components of $p^{-1}((-1,1)^n)$. Then, the function $f\coloneqq g\circ(p|_{\overline{V}})^{-1}$ has the required properties.
\end{proof}


\section{Nash globally subanalytic arc-symmetric sets}
\label{sec:Nash-glob-sub}

Let $\Omega$ be a real analytic manifold. Let $E\subset\Omega$ be a subanalytic set, and let $x\in\Omega$. Suppose first that $E$ is of pure dimension $k$. We say that $E$ is \emph{Nash at $x$}, when there exists a neighbourhood $U$ of $x$ in $\Omega$ and an analytic set $S\subset U$, of dimension $k$, such that $E\cap U\subset S$. Suppose now that $E$ is not pure-dimensional. We say that $E$ is Nash at $x$, when $E$ is a finite union of pure-dimensional subanalytic sets each of which is Nash at $x$. We say that $E$ is a \emph{Nash subanalytic} set in $\Omega$, when $E$ is Nash at each point of $\Omega$.

\begin{theorem}
\label{thm:Nash-is-semi}
Let $\Omega$ be a connected real analytic submanifold of $\R^n$, and let $E\in\AR(\Omega)$.
If $E$ is Nash subanalytic, then $E$ is semianalytic in $\Omega$.
\end{theorem}

\begin{proof}
We proceed by induction on dimension of $E$. If $\dim{E}\leq0$, then $E$ is a finite set and hence the zero locus of a polynomial function on $\R^n$ (restricted to $\Omega$). Suppose then that $\dim{E}=k>0$, and that every Nash globally subanalytic arc-symmetric subset of a connected real analytic submanifold $\Theta$ of $\R^n$ of dimension less than $k$ is semianalytic in $\Theta$.

Since $E$ is closed in $\Omega$, it suffices to show that $E$ is semianalytic in a sufficiently small neighbourhood of each of its points.
Pick an arbitrary $\xi\in E$. Without loss of generality, we may assume that $\dim_\xi E=k$. For $j=0,\dots,k$, let $E^{(j)}=\{x\in E:\dim_x{E}=j\}$, and let $E^*=\bigcup_{j=0}^{k-1}E^{(j)}$. By Nashness of $E$, one can choose an open ball $U$ centered at $\xi$, such that there exist globally subanalytic real analytic sets $R$ and $T$ in $U$, with the following properties: $\dim{R}=k$, $E\cap U\subset R$, $\dim{T}<k$, $E^*\cap U\subset T$. Moreover, after shrinking $U$ if needed, there is a globally subanalytic real analytic set $S\subset R$, of dimension $\dim{S}<k$, such that $R\setminus S$ is smooth of pure dimension $k$.

Let now $\{S_\lambda\}_{\lambda\in\Lambda}$ be a semianalytic stratification of $U$, compatible with $R$, $S$, and $T$. Assume $\Lambda$ is finite, after further shrinking $U$ if needed. Let $\Lambda_R=\{\lambda\in\Lambda:S_\lambda\subset R\}$. Note that, for every $\lambda\in\Lambda_R$, we have $S_\lambda\subset E$, or else $E\cap S_\lambda$ is nowhere dense in $S_\lambda$ and $\dim_xE<k$ for all $x\in E\cap S_\lambda$, by Lemma~\ref{lem:dim-drop}. Let $\Lambda_E=\{\lambda\in\Lambda_R:S_\lambda\subset E\}$. We thus obtain
\begin{multline}
\label{eq:complement}
\overline{R\setminus E}=\overline{(\bigcup_{\lambda\in\Lambda_R}S_\lambda)\setminus E}=\overline{\bigcup_{\lambda\in\Lambda_R}(S_\lambda\setminus E)}\\
=\overline{\bigcup_{\lambda\in\Lambda_R\setminus\Lambda_E}\!\!\!(S_\lambda\setminus E)}=\bigcup_{\lambda\in\Lambda_R\setminus\Lambda_E}\!\!\!\overline{S_\lambda\setminus E}=\bigcup_{\lambda\in\Lambda_R\setminus\Lambda_E}\!\!\!\overline{S_\lambda}\,.
\end{multline}
By a well known characterization of semianalyticity (see, e.g., \cite[Thm.\,2.13]{BM1}), to prove that $E\cap U$ is semianalytic it suffices to show that $(E\cap U)\setminus\Int_R(E)$ is semianalytic of dimension less than $k$, where $\Int_R(E)=E\setminus\overline{R\setminus E}$ is the interior of $E$ in $R$.
It follows from~\eqref{eq:complement} that
\[
(E\cap U)\setminus\Int_R(E)=E\cap\overline{R\setminus E}=E\,\cap\!\!\!\bigcup_{\lambda\in\Lambda_R\setminus\Lambda_E}\!\!\!\overline{S_\lambda}=E\cap(S\cup T)\,.
\]
The latter is a Nash $\AR$-closed subset of $U$ of dimension less than $k$, and hence semianalytic by inductive hypothesis.
\end{proof}

For the next result, we shall adapt the concept of a \emph{$q$-grid} from \cite{BP}. Given a positive integer $q$, a $q$-grid centered at $\xi=(\xi_1,\dots,\xi_n)\in\R^n$ is defined as the union of coordinate hyperplanes
\[
\Sigma=\bigcup_{j=1}^n\,\bigcup_{k\in\Z}\{x_j=\xi_j+\;k/q\}\,.
\]
Let $\{C_\lambda\}_{\lambda\in\Lambda}$ denote the family of open cubes induced by $\Sigma$ (i.e., the connected components of $\R^n\setminus\Sigma$).
We say that $\Sigma$ is \emph{subordinate} to an open cover $\mathcal{U}=\{U_\iota\}_{\iota\in I}$ of $\R^n$, when for every $\lambda\in\Lambda$ there exists $\iota\in I$ with $\overline{C}_\lambda\subset U_\iota$. Given a subanalytic set $E\subset\R^n$, we say that $\Sigma$ is \emph{in general position} with respect to $E$, when $\dim_x\!{\Sigma\cap E}<\dim_x\!{E}$ for every $x\in E$.

\begin{remark}
\label{rem:q-grid}
Let $\Omega$ be a relatively compact open set in $\R^n$.
It is evident from the proof of \cite[Lem.\,2.4]{BP} that for every locally finite open cover $\mathcal{U}$ of $\R^n$ and for every closed globally subanalytic set $E\subset\R^n$ there exists a positive integer $q$ and a point $\xi\in\R^n$ such that, up to a linear coordinate change in $\R^n$, the $q$-grid $\Sigma$ centered at $\xi$ is subordinate to $\mathcal{U}$ on $\Omega$ and in general position with respect to $E$.
\end{remark}

\begin{theorem}
\label{thm:AR-zero-set}
Let $E$ be a Nash globally subanalytic arc-symmetric set in $\R^n$. Then, for every relatively compact open $\Omega\subset\R^n$, there exists a continuous globally subanalytic function $f:\Omega\to\R$ and a simple normal crossings divisor $\Sigma$ in $\R^n$, such that
\begin{itemize}
\item[(i)] $\dim_x\Sigma\cap E<\dim_xE$ for all $x\in E$
\item[(ii)] $f$ is arc-analytic on $\Omega\setminus\Sigma$, and
\item[(iii)] $E\cap\Omega=f^{-1}(0)$.
\end{itemize}
\end{theorem}

\begin{proof}
By Theorem~\ref{thm:Nash-is-semi}, for every point $x\in\R^n$, there exists a relatively compact connected open neighbourhood $U^x$ of $x$ in $\R^n$ such that $E\cap U^x\in\AR_C(U^x)$. Let $\mathcal{U}=\{U_\iota\}_{\iota\in I}$ be a locally finite subcover of $\R^n$ chosen from the open cover $\{U^x\}_{x\in E}$. Let $\Omega$ be an arbitrary relatively compact nonempty open set in $\R^n$.
By Remark~\ref{rem:q-grid}, there is a $q$-grid $\Sigma'$ in $\R^n$ such that $\Sigma'$ is in general position with respect to $E$ and the covering $\mathcal{C}=\{\overline{C}_\lambda\}_{\lambda\in\Lambda}$ of $\Omega$ by the closed cubes induced by $\Sigma'$ is subordinate to $\mathcal{U}$.

Given $\lambda\in\Lambda$, pick $\iota\in I$ such that $\overline{C}_\lambda\subset U_\iota$. By Corollary~\ref{cor:ARC-on-cubes}, there is a continuous globally subanalytic function $f_\lambda:\overline{C}_\lambda\to\R$ such that $f_\lambda$ is arc-analytic on $C_\lambda$ and $f_\lambda^{-1}(0)=(E\cup\Sigma')\cap\overline{C}_\lambda$. Define $f_{\Sigma'}:\Omega\to\R$ as $f_{\Sigma'}\coloneq\bigcup_{\lambda\in\Lambda}f_\lambda|_{\overline{C}_\lambda}$. This $f_{\Sigma'}$ is continuous, globally subanalytic, arc-analytic outside $\Sigma'$, and $f^{-1}_{\Sigma'}(0)=E\cup\Sigma'$.

To complete the proof, note that the $q$-grid $\Sigma'$ can be chosen so that a $q$-grid $\Sigma'_l\coloneq\{(x_1+\frac{l}{kq},\dots,x_n+\frac{l}{kq}):(x_1,\dots,x_n)\in\Sigma'\}$ is also subordinate to $\mathcal{U}$ on $\Omega$ and in general position with respect to $E$, for some $k\in\Z_+$ and all $l\in\{1,\dots,n\}$. Let $f_{\Sigma'_l}:\Omega\to\R$ be the corresponding continuous  globally subanalytic functions with $f^{-1}_{\Sigma'_l}(0)=E\cup\Sigma'_l$ constructed as above. Then, the function $f\coloneq f_{\Sigma'}+f_{\Sigma'_1}+\dots+f_{\Sigma'_n}$ is continuous, globally subanalytic, and arc-analytic outside the simple normal crossings divisor
\[
\Sigma\coloneq\{(x_1+\frac{l_1}{kq},\dots,x_n+\frac{l_n}{kq}):(x_1,\dots,x_n)\in\Sigma',\, l_1,\dots,l_n\in\{0,\dots,n\}\}\,.
\]
Moreover, $f^{-1}(0)=E$.
\end{proof}

\bibliographystyle{amsplain}

\end{document}